\documentclass{amsart}
\usepackage[utf8]{inputenc}
\usepackage{amsmath}
\usepackage{hyperref}
\usepackage{dynkin-diagrams}
\usepackage[shortlabels]{enumitem}
\usepackage{color}

\usepackage{multirow}

\newtheorem{theorem}{Theorem}[section]
\newtheorem{lemma}[theorem]{Lemma}

\theoremstyle{definition}

\theoremstyle{remark}
\newtheorem{remark}[theorem]{Remark}

\numberwithin{equation}{section}

\begin{document}

\title[An Exceptional Combinatorial Sequence]{An Exceptional Combinatorial Sequence and Standard Model Particles}

\author{Benjamin Nasmith}
% \address{}
% \curraddr{Department of Mathematics and Statistics, Case Western Reserve University, Cleveland, Ohio 43403}
% \email{benjamin.nasmith@rmc.ca}
% \thanks{Department of Physics and Space Science, Royal Military College of Canada, Kingston, ON}

\date{\today}

% \dedicatory{Comprehensive exam preparation notes.}


\begin{abstract}
    Three-graded root systems can be arranged into nested sequences. One exceptional sequence provides a natural means to recover some structures and symmetries familiar in the context of particle physics.
\end{abstract}

\maketitle

\section{Introduction}

We can succinctly describe many features of both Lie and Jordan structures in algebra and geometry using {root systems}. The following sequence of root systems has a number of exceptional properties:
\begin{align}
    \label{mainsequence}
    E_7 \rightarrow
    E_6 \rightarrow
    D_5 \rightarrow
    A_4 \rightarrow
    A_1 \times A_2. \tag{$\star$}
\end{align}
The final root system and nesting in this sequence, $A_4 \rightarrow A_1\times A_2$, corresponds to the Lie group of the standard model of particle physics: $\mathrm{U(1)}\times\mathrm{SU}(2)\times\mathrm{SU}(3)$. 
The third and fourth root systems correspond to two well-studied grand unification theories: the $\mathrm{Spin}(10)$ and $\mathrm{SU}(5)$ theories. 
This note describes some special properties of this sequence of root systems and explains how it affords a natural representation of all three generations of standard model fermions. 

\section{Star-Closed Line Systems}

Consider the three axes of a regular hexagon in $\mathbb{R}^2$. These lines have the special property that the angle between any two of the three is $60$ degrees. That is, the three axes of a regular hexagon are a system of \textbf{equiangular lines}. It turns out that for any system of equiangular lines in $\mathbb{R}^d$, the number of lines $n$ must satisfy $n \le \binom{d+1}{2}$. 
The number $\binom{d+1}{2}$ is called the \textbf{absolute bound} on the number of equiangular lines in $d$-dimensions \cite[chap. 11]{godsil_algebraic_2001}. The three axes of the hexagon meet this absolute bound in $d=2$ dimensions. The only other known examples of equiangular lines at the absolute bound consist of the axes of an icosahedron in $d=3$, a $28$ line system for $d=7$, and a $276$ line system for $d=23$. Any further examples, if they exist, will occur in $d \ge 119$ \cite[1402]{bannai_survey_2009}.

In what follows we will refer to three lines at $60$ degrees as a \textbf{star}. The star is the smallest system of equiangular lines at the absolute bound, and stars are responsible for an abundance of rich structures in algebra and combinatorics. Examples of structures that can be constructed from stars include root systems, root lattices, Lie algebras, Jordan grids, Jordan triple systems, Jordan algebras, and many interesting spherical and projective $t$-designs. 

We will focus for the moment on line systems of type $(0,1/2)$. A \textbf{line system of type} $(a_1, a_2, \ldots, a_n)$ is a finite set of lines through the origin of a real vector space (equivalently, points in a real projective space) such that the Euclidean inner product of any two unit vectors spanning distinct lines satisfies $|\cos\theta| \in \{a_1, a_2, \ldots, a_n\}$. Line systems of type $(0,1/2)$ are studied in \cite{cameron_line_1976}, while line systems of types $(0,1/3)$ and $(0,1/2,1/4)$ are studied in \cite{shult_near_1980}. In what follows we will refer to line systems of type $(0,1/2)$ simply as line systems. That is, we will take a \textbf{line system} to be a set of lines in a real vector space such that any two lines in the system are either orthogonal or at $60$ degrees. 

Each pair of non-orthogonal lines in a line system defines a unique coplanar line that is at $60$ degrees to both members of the pair. Three lines at $60$ degrees form a \textbf{star}, and any two members of a star defines the third member. Using this concept, we can compute the \textbf{star-closure} of a line system by adding to the line system any missing third lines defined by any nonorthogonal pair. When a line system is equal to its own star-closure, it is a \textbf{star-closed line system}. When a line system cannot be partitioned into two mutually orthogonal subsets, it is an \textbf{indecomposable line system}. Finally, a \textbf{star-free line system} is a line system without stars, in which any three mutually non-orthogonal lines span a vector space of dimension three.

The indecomposable star-closed line systems are classified in \cite{cameron_line_1976}. The classification makes heavy use of the following lemma:
\begin{lemma}
    Let $L$ be a line system and let $S \subset L$ be a star. Then each line in $L \setminus S$ is orthogonal to either $1$ or $3$ members of $S$.
\end{lemma}
That is, for line system $L$ containing star $S$, we can partition the lines of $L$ into $S$, lines orthogonal to $S$, and three sets of lines orthogonal to just one member of $S$. We may call this partition the \textbf{star-decomposition} of line system $L$ with respect to star $S \subset L$. 
That is, for $S = \{a,b,c\}$ we can write $L = S~\dot{\cup}~A~\dot{\cup}~B~\dot{\cup}~C~\dot{\cup}~D$, where $A$ is the set of lines in $L$ orthogonal to just $a$, $B$ orthogonal to just $b$, $C$ orthogonal to just $c$, and $D$ orthogonal to all three lines of $S$. 
We will see below that the physics concepts of particle \textit{colour} and \textit{generation} can be recovered from the combinatorial concept of line system star-decomposition. 

When $L$ is an indecomposable star-closed line system, we can say a number of helpful things about subsets of lines in the star-decomposition of $L$, as developed in \cite[chap. 12]{godsil_algebraic_2001}. First, $L$ is the star-closure of $S~\dot{\cup}~A$. Second, the set $A$ does not contain any stars and we can find a set of vectors spanning $A$ with all non-negative inner products. Third, any pair of orthogonal lines in $A$ belongs to a set of three mutually orthogonal lines in $A$, called a \textbf{triad}. Fourth, the triads in $A$ always form the ``lines'' of a generalized quadrangle. So the task of classifying indecomposable star-closed line systems is equivalent to the task of classifying the (possibly trivial) generalized quadrangle structures with ``lines'' of size $3$ on the set $A$ of the star-decomposition of that system. 

A \textbf{generalized quadrangle} is a point-line incidence structure such that the bipartite incidence graph has diameter 4 and girth 8. 
We denote by $GQ(s,t)$ a generlized quadrangle in which each ``line'' contains $s+1$ ``points'' and each ``point'' belongs to $t+1$ ``lines''. 
In terms of $A$, the ``points'' are the lines of $A$ and the ``lines'' are the orthogonal triads of $A$. 
We will see below that the lines corresponding to a single generation of particles define a generalized quadrangle $GQ(2,2)$ with automorphism group $S_6$ (the only symmetric group with a non-trivial outer automorphism).

We will say that the lines of $A$ \textbf{represent} graph $G$ if we can find a vector on each line of $A$ such that the Gram matrix of these vectors, apart from the diagonal entries, is the adjacency matrix of $G$. In the case of star-free $A$, the graph $G$ has the lines of $A$ for vertices and two vertices adjacent if and only if they are non-orthogonal lines. The vertices of this graph and the maximal independent sets must form the ``points'' and ``lines'' of a generalized quadrangle, albeit a possibly trivial one. 
This restriction on the possible structure of $A$ yields the classification of indecomposable star-closed line systems. For more details on the following theorem, see \cite[chap. 12]{godsil_algebraic_2001}.
\begin{theorem}
    \cite{cameron_line_1976} Every indecomposable star-closed line system is the star-closure of a system of lines $S~\dot{\cup}~A$, where $S$ is a star and $A$ is a star-free set of lines orthogonal to just one line in $S$, and where  $A$ represents graph $G$ with maximal independent sets forming a generalized quadrangle:
    \begin{enumerate}[(a)]
        \item $\overline{A_n}$ for $G$ the complete graph $K_{n-2}$,
        \item $\overline{D_n}$ for $G$ the cocktail party graph $CP(n-3)$ plus an isolated vertex,
        \item $\overline{E_6}$ for $G$ the unique strongly regular graph with parameters $(9,4,1,2)$,
        \item $\overline{E_7}$ for $G$ the unique strongly regular graph with parameters $(15,8,4,4)$,
        \item $\overline{E_8}$ for $G$ the unique strongly regular graph with parameters $(27,16,10,8)$.
    \end{enumerate}
    \label{starclosedclassification}
\end{theorem}
Here we denote by $\overline{\Phi}$ a star-closed line system and by $\Phi$ the set of length $\sqrt{2}$ vectors spanning the lines of $\overline{\Phi}$. 
As the labels above suggest, the star-closed line systems are precisely the lines spanned by more familiar the \textbf{simply-laced root systems}, the root systems with all equal-length roots. Note that the standard terminology is such that an \textit{indecomposable} line system $\overline{\Phi}$ corresponds to an \textit{irreducible} root system $\Phi$.

\begin{remark}
    Not all irreducible root systems are simply-laced. That is, there are irreducible root systems of types $B_n$, $C_n$, $G_2$, and $F_4$ that include roots of two different lengths. 
    We can recover these systems from line systems via the root lattices of the corresponding simply-laced root systems. Put another way, every root lattice is also the root lattice of a simply-laced root system \cite[99]{conway_sphere_2013}. 
    
    First, suppose that we have an $\overline{A_1^n}$ star-closed system of lines. 
    This is simply a set of $n$ mutually orthogonal lines in $\mathbb{R}^n$.
    Take the vectors of length $\sqrt{2}$ spanning these lines. 
    These vectors span the root lattice of type $\mathbb{Z}^n$. 
    The second layer of that lattice (the lattice points at the second shortest distance to the origin with the roots forming the first layer) is a root system of type $D_n$, containing roots of length $2$. 
    The sum of these $A_1^n$ roots of length $\sqrt{2}$ and $D_n$ roots of length $2$ is a $B_n$ root system.
    
    Second, suppose that we have a $\overline{D}_n$ star-closed system of lines. 
    Take the vectors of length $\sqrt{2}$ spanning these lines to obtain a $D_n$ root system spanning a $D_n$ root lattice. 
    There exists in the second layer of the $D_n$ lattice a subset of vectors that both spans an $\overline{A_1^n}$ set of lines and identifies additional reflection symmetries of the underlying $D_n$ system. 
    If we include these vectors, we obtain a $C_n$ root system. 
    
    Finally, we obtain the $G_2$ roots by taking the first two layers of the lattice defined by $\overline{A}_2$, and the $F_4$ roots by taking the first two layers of the lattice defined by $\overline{D}_4$.
\end{remark}

\section{Nested Sequences of Binary Decompositions}

We have seen that any indecomposable star-closed line system admits a star decomposition. 
Apart from $\overline{E}_8$, it turns out that every indecomposable star-closed line system also admits at least one \textbf{binary decomposition}, namely a partition $\overline{\Phi} = \overline{\Phi}_0 ~\dot{\cup}~\overline{\Phi}_1$ such that $\overline{\Phi}_0$ is \textit{star-closed}, $\overline{\Phi}_1$ is \textit{star-free}, and $\overline{\Phi}$ is the star-closure of the star-free component $\overline{\Phi}_1$.
We can characterize binary decompositions in terms of 3-gradings of simply-laced root systems, since each star-closed line system corresponds to a simply-laced root system. 
Following \cite[168]{loos_locally_2004}, we define a \textbf{3-grading on a root system} $\Phi$ as a partition,
\begin{align*}
     \Phi = \Phi_{-1}~\dot{\cup}~\Phi_{0}~\dot{\cup}~\Phi_{1},
\end{align*}
such that,
\begin{align*}
    \Phi \cap (\Phi_{a} + \Phi_{b} ) \subset \Phi_{a+b}, \quad \Phi_c = \varnothing \text{ for } c \ne -1,0,1,
\end{align*}
and also,
\begin{align*}
    \Phi \cap (\Phi_1 - \Phi_1) = \Phi_0.
\end{align*}
That is, if the difference between any two roots in $\Phi_1$ is also a root, then it is a root in $\Phi_0$. Also, every root in $\Phi_0$ is the difference of some two roots in $\Phi_1$.  
Since every $3$-grading corresponds to a homomorphism from the corresponding root lattice to the grading group, we have $\Phi_{-1} = - \Phi_{1}$. This means that we can recover the entire root system from the $\Phi_{1}$ piece alone, as linear combinations of roots in $\Phi_{1}$. 
In particular, the 3-grading defined by $\Phi_1$ defines a star-free set of lines $\overline{\Phi}_1$, spanned by the roots of $\Phi_1$. 
Just as we can recover $\Phi$ from $\Phi_1$ by familiar Weyl reflections, so also can we recover $\overline{\Phi}$ from $\overline{\Phi}_1$ by star-closure. 

% Suppose that $R$ and $S$ are two $3$-graded root systems. We say that $R$ and $S$ form an \textbf{embedding} of $3$-gradings when $S_i \subset R_i$ for $i = -1,0,1$. We say instead that $R$ and $S$ form a \textbf{nesting} of $3$-gradings when $S = R_0$. 
%In what follows, we denote a nesting of root systems with a diagram of the form $R \rightarrow S$. 
% Nestings and embeddings are closely related, as we will see below. 

The \textbf{coweight} of a root system $\Phi$ is a vector $q$ such that for each root $\alpha$ in $\Phi$, the \textbf{Euclidean inner product} $\langle \alpha, q\rangle$ is an integer. 
In general, a $\mathbb{Z}$-grading on a root system $\Phi$ can be identified with some coweight $q$ as follows \cite[166]{loos_locally_2004}:
\begin{align*}
    \Phi_i = \Phi_i(q) = \left\{\alpha \in \Phi \mid \langle \alpha, q \rangle = i \in \mathbb{Z}\right\}.
\end{align*}
The coweights responsible for $3$-gradings are the \textbf{minuscule coweights} \cite[61]{loos_locally_2004}. 
That is, a minuscule coweight of $\Phi$ is a vector $q$ such that $\langle \alpha,q\rangle = -1, 0 ,1$ for all roots $\alpha$. 
These facts can be used to show that the possible $3$-gradings on connected root systems are classified using the weighted Coxeter-Dynkin diagrams shown in Fig. \ref{3gradings}. In each case, we obtain the $3$-grading of a irreducible root system $\Phi$ by identifying the $\Phi_0$ component as the root subsystem with its Coxeter-Dynkin diagram given by the dark vertices \cite[171]{loos_locally_2004}.
\begin{figure}[!h]
    \centering
    \begin{tabular}{l l l}
        \hline
        \hline
        $3$-Grading Name & Coxeter-Dynkin Diagram & $\Phi \xrightarrow{|\Phi_1|} \Phi_0$
        \\
        \hline
        rectangular & $\dynkin{A}{**.o.*}$
        & $A_{p+q-1} \xrightarrow{p q} A_{p-1}\times A_{q-1}$
        \\
        hermitian & $\dynkin{C}{**.*o}$ 
        & $C_n \xrightarrow{\binom{n+1}{2}} A_{n-1}$
        \\
        odd quadratic & $\dynkin{B}{o*.**}$ 
        & 
        $B_n \xrightarrow{2n-1} B_{n-1}$
        \\
        even quadratic & $\dynkin{D}{o*.****}$     
        &
        $D_n \xrightarrow{2(n-1)} D_{n-1}$  
        \\
        alternating & $\dynkin{D}{**.***o}$ 
        &
        $D_n \xrightarrow{\binom{n}{2}} A_{n-1}$
        \\
        Albert & $\dynkin{E}{******o}$ 
        &
        $E_7 \xrightarrow{27} E_6$
        \\
        bi-Cayley & $\dynkin{E}{o*****}$ 
        &
        $E_6 \xrightarrow{16} D_5$
        \\
        \hline
        \hline
\end{tabular}
    \caption{The $3$-gradings on finite irreducible root systems.}
    \label{3gradings}
\end{figure}

We see from Fig. \ref{3gradings} that root systems of types $B_n$, $C_n$, $E_6$, and $E_7$ only admit one possible type of $3$-grading. Root systems of types $A_n$ and  $D_n$ admit multiple possible $3$-gradings. In the case of $A_n$ root systems, there are $\lfloor (n+1)/2 \rfloor$ possible rectangular $3$-gradings. In the case of $D_n$ root systems, there is a quadratic $3$-grading and an alternating $3$-grading. Root systems of types $E_8$, $G_2$, and $F_4$ do not admit a 3-grading. In each case, we need only identify the $\Phi_0$ component to identify the $3$-grading.

% Suppose that $R$ and $S$ are two $3$-graded root systems. We say that $R$ and $S$ form an \textbf{embedding} of $3$-gradings when $S_i \subset R_i$ for $i = -1,0,1$. We say instead that $R$ and $S$ form a \textbf{nesting} of $3$-gradings when $S = R_0$. 
We define a \textbf{sequence of nested} 3\textbf{-gradings} as a sequence of root systems $\Phi^{(n)} \subset \Phi^{(n+1)}$ such that $\Phi^{(n)} = \Phi_0^{(n+1)}$. We denote such a sequence using a diagram of the form,
\begin{align*}
    \cdots \rightarrow \Phi^{(n+1)} \xrightarrow{|\Phi_1^{(n+1)}|}  \Phi^{(n)} \xrightarrow{|\Phi_1^{(n)}|} \Phi^{(n-1)} \rightarrow \cdots
\end{align*}
The weight of the arrow is the dimension of the $1$-part of the $3$-grading it represents. Fig. \ref{GradingMesh} illustrates the structure of sequences of nested $3$-gradings for the simply-laced root systems of rank $7$ or less. 
Multiple sequences can pass through a single root system. For instance, from Fig. \ref{GradingMesh} we see that both $D_n \rightarrow D_{n-1} \rightarrow A_{n-2}$ and $D_n \rightarrow A_{n-1} \rightarrow A_{n-2}$ represent possible nestings of $3$-gradings containing both $D_n$ and $A_{n-2}$. The diagram could be extended to the upper-right by including higher rank root systems, adding the arrows $D_8 \xrightarrow{14} D_7$, $A_8 \xrightarrow{8} A_7$, $D_8 \xrightarrow{28} A_7$, and so on. 

\begin{figure}[!ht]
    \centering
    \includegraphics{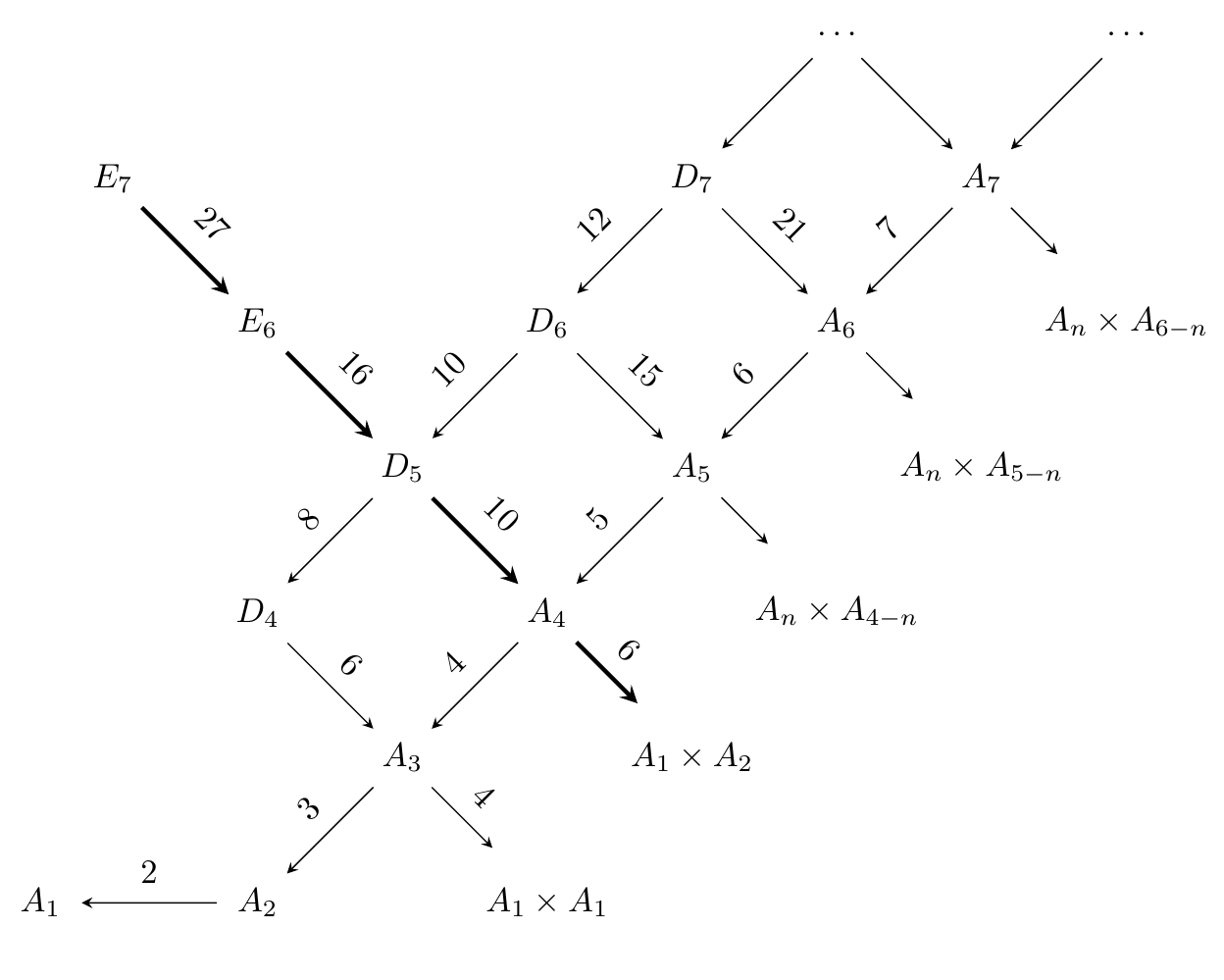}
    \caption{Nested $3$-gradings of simply-laced irreducible root systems in $\mathbb{R}^7$.}
    \label{GradingMesh}
\end{figure}

Sequences of nested $3$-gradings on root systems correspond to sequences of nested binary decompositions on line systems, and vice versa. The mesh of available $3$-gradings on irreducible root systems shown in Fig. \ref{GradingMesh} also applies to indecomposable star-closed line systems to describe the available binary decompositions. By working with line systems, we can better appreciate that the combinatorial properties of $\Phi_1$ and $\Phi_{-1}$ are equivalent. Indeed, the lines defined by $\Phi_1$ are precisely the same lines defined by $\Phi_{-1}$.

The exceptional sequence \eqref{mainsequence}, bolded in Fig. \ref{GradingMesh}, is one example of a sequence of nested 3-gradings of simply-laced root systems, or equivalently a sequence nested binary decompositions of star-closed line systems. 

\section{The Exceptional Sequence}

We now identify some special properties of sequence \eqref{mainsequence}, in comparison to all other possible  sequences of nested binary decompositions, as illustrated in Fig. \ref{GradingMesh}. 

First, sequence \eqref{mainsequence} begins with $E_7$, which is the only indecomposable star-closed line system (or irreducible simply-laced root system) that admits a binary decomposition but is not embedded in another line system as the zero-component of a binary decomposition. 
That is, any sequence of nested binary decompositions can be extended further to the left unless it begins with $E_7$. 
So sequences that begin with $E_7$ and end in either $A_1$, $A_1\times A_1$, or $A_1 \times A_2$ are unique in that they cannot be made any longer by extended to the left or the right.

Second, sequence \eqref{mainsequence} is a local sequence in the following sense. 
For a binary decomposition $\overline{\Phi} = \overline{\Phi}_1~\dot{\cup}~\overline{\Phi}_0$, we can define a \textbf{binary decomposition graph} $G$ with the lines of $\overline{\Phi}_1$ for vertices and all pairs of nonorthogonal lines for edges. 
Using this definition, we can assign a graph to each binary decomposition, or arrow, in a nested sequence.
The graph of a binary decomposition for $\overline{\Phi}$ indecomposable is always vertex-transitive. This means that there is a unique \textbf{local subgraph} of $G$, the induced subgraph on the neighbours of any given point.  
We will say that a sequence of nested binary decompositions is a \textbf{local sequence} when the binary decomposition graph of each arrow is isomorphic to the local subgraph of the binary decomposition graph in the preceding arrow.
The possible local sequences beginning with indecomposable star-closed line systems are as follows:
\begin{align*}
    \cdots &\rightarrow A_n \rightarrow A_{n-1} \rightarrow \cdots \rightarrow A_2 \rightarrow A_1, \\
    \cdots &\rightarrow D_n \rightarrow D_{n-1} \rightarrow \cdots \rightarrow D_4 \rightarrow A_3 \rightarrow A_1 \times A_1, \\
    D_n &\rightarrow A_{n-1} \rightarrow A_1 \times A_{n-3}, \\
    E_7 &\rightarrow
    E_6 \rightarrow
    D_5 \rightarrow
    A_4 \rightarrow
    A_1 \times A_2.
\end{align*}
If we restrict ourselves to local sequences that cannot be embedded in a longer sequence, then the exceptional sequence \eqref{mainsequence} is the only one with this property, since it is the only local sequence that begins with $E_7$. 

Third, sequence \eqref{mainsequence} is a maximal sequence in the following sense. 
We say that a sequence of nested binary decompositions is a \textbf{maximal sequence} when the path of the sequence through the possible binary decompositions, shown in Fig. \ref{GradingMesh}, is such that the largest $\overline{\Phi}_1$ component is chosen in each case. That is, a maximal sequence always follows the highest weight arrows from a given starting point in Fig. \ref{GradingMesh}.

\begin{theorem}
    The sequence \eqref{mainsequence} is the unique local and maximal sequence of nested 3-gradings (or binary decompositions) that cannot be embedded in a longer sequence.     
\end{theorem}

\begin{proof}
    Any sequence that cannot be embedded in a longer sequence begins with $E_7$. The only local sequence beginning with $E_7$ is the sequence \eqref{mainsequence}. Likewise, the only maximal sequence beginning with $E_7$ is the sequence \eqref{mainsequence}.
\end{proof}

\begin{remark}
The minuscule coweights of $E_7$ span the unique system of $28$ equiangular lines in $\mathbb{R}^7$ that attain the absolute bound $\binom{7+1}{2}$ described earlier. 
% If we project any of these $E_7$ minuscule coweights onto the subspace spanned by a subsystem of $E_7$, we obtain a minuscule coweight for that subsystem. 
% Since minuscule coweights define 3-gradings on root systems, we can describe nested sequences that begin with $E_7$ in terms of subsets of the $28$ equiangular lines at the absolute bound in $\mathbb{R}^7$. 
% That is, we can project a sequence of $E_7$ minuscule coweights onto the subspace of subsystems defined by the previous minuscule coweight in the sequence, and thereby identify additional
By \textbf{accute minuscule coweights} we mean a set of minuscule coweights with positive pairwise inner product. 
Recall that $\overline{E}_6$ is constructed by taking the lines of $\overline{E}_7$ orthogonal to a single member of the $28$ equiangular lines.
Likewise, $\overline{D}_5$, $\overline{A}_4$, and $\overline{A_1 \times A_2}$ are constructed as the lines of $\overline{E}_7$ orthogonal to a pair, triple, and quadruple of accute minuscule coweights, and the subset of the $28$ equiangular lines they span. 
So we can also understand the sequence \eqref{mainsequence} by taking roots orthogonal to successively larger sets of accute minuscule coweights of $E_7$.
\end{remark}

\section{Lie Algebras of Star-Closed Line Systems}

Certain important Lie and Jordan structures correspond to star-closed line systems and binary decompositions. Indeed, all Jordan triple systems are constructed from 3-gradings on root systems, or equivalently from binary decompositions on line systems.
In what follows we focus on Lie algebras, given their direct application to particle physics. Even so, many of the structures described below could be constructed using the Jordan triple systems corresponding to the $3$-graded Lie algebra in question.

% This section describes a construction of the Lie algebra of the standard model of particle physics acting on a $96$ dimensional vector space with eigenvalues corresponding to the three generations of standard model fermions. 
A \textbf{Lie algebra} is a vector space $\mathfrak{g}$ with product $[x,y]$ such that $[x,x] = 0$ and $[[x,y],z] + [[y,z],x] + [[z,x],y] = 0$ for all vectors $x,y,z$. 
Lie algebras are non-associative in general, and we say that a Lie algebra is \textbf{abelian} when $[x,y]=0$ for all $x,y$.  
We can construct certain important Lie algebras (the semi-simple ones) using root systems, including the simply-laced root systems corresponding to star-closed line systems. 
\begin{theorem}
    \cite[42-43]{carter_simple_1972} Let $\Phi$ be an irreducible root system. Then there exists, up to Lie algebra isomorphism, a simple Lie algebra $\mathfrak{g}$ over $\mathbb{C}$ with a Chevalley basis.  
    \label{rootsandLiealgebras}
\end{theorem}
That is, given root system $\Phi$, there is a $\Phi$-graded Lie algebra of the form,
\begin{align*}
    \mathfrak{g} = \mathfrak{h} \oplus \bigoplus_{r \in \Phi} \mathfrak{g}_r.
\end{align*}
This is called the \textbf{Cartan grading} of Lie algebra $\mathfrak{g}$.
Here  $\mathfrak{h}$ is a \textbf{Cartan subalgebra} of $\mathfrak{g}$ and $\mathfrak{g}_r$ are the \textbf{root spaces} of the decomposition. 
The dimension of $\mathfrak{h}$ is equal to the dimension of the space $\mathbb{R}^n$ spanned by the roots $\Phi$, whereas the dimension of each root space  $\mathfrak{g}_r$ is $1$. The \textbf{rank} of the Lie algebra is the dimension of $\mathfrak{h}$. 
The Cartan subalgebra $\mathfrak{h}$ has basis $h_r$, where $r$ is each simple root of $\Phi$ (corresponding to the vertices of the Coxeter-Dynkin diagram of $\Phi$). 
Each subalgebra $\mathfrak{g}_r$ is spanned by basis vector $e_r$, where $r$ is a root in $\Phi$. 
For any $x$ not in $\Phi$ we have $\mathfrak{g}_x = 0$. 
The products involving the Cartan subalgebra $\mathfrak{h}$ are defined entirely in terms of the geometry of the roots $r,s$ in $\Phi$:
\begin{align*}
    [h_r, h_s] = 0, &&
    [h_r, e_s] = \frac{2 \langle r, s \rangle}{\langle r, r\rangle} e_s, &&
    [e_r, e_{-r}] = h_r. 
\end{align*}
Here $\langle r, s\rangle$ denotes the standard Euclidean inner product between vectors $r,s$ in $\mathbb{R}^n$ (where $\mathrm{dim}_\mathbb{C} (\mathfrak{h}) = n$).
% For this reason, we sometimes write $h_r = 2r/\langle r, r\rangle$. 
Products of the root spaces of two linearly independent roots are defined by,
\begin{align*}
    [e_r, e_s] = N_{r,s} e_{r+s}.
\end{align*}
The structure constants $N_{r,s}$ can be fixed without loss of generality to define the Chevalley basis, as described in \cite[56-57]{carter_simple_1972}.
Theorem \ref{rootsandLiealgebras} applies to all irreducible root systems. 
In what follows we only make use of the cases involving simply-laced root systems, which are listed in Fig. \ref{simpleLieAlgebras} \cite[43]{carter_simple_1972}. 
\begin{figure}[!ht]
    \centering
    \begin{tabular}{c c c c c c}
    \hline\hline
    \text{Type} & $\mathfrak{g}$ & $\mathrm{dim~}\mathfrak{g}$ & $\mathrm{rank~}\mathfrak{g}$ & $\lvert\Phi\rvert$ &  \text{Dynkin diagram} \\
    \hline 
    $A_n~(n \ge 1)$ & $\mathfrak{sl}_{n+1}$ &   $n(n+2)$ & $n$ & $n(n+1)$ & $\dynkin{A}{}$ \\
    % B_n~(n \ge 2) & \mathfrak{so}({2n+1}) &  n(2n+1) & n &       2 n^2          & \dynkin{B}{} \\
    % C_n~(n \ge 3) & \mathfrak{sp}({2n}) &  n(2n+1) & n &       2 n^2          & \dynkin{C}{} \\
    $D_n~(n \ge 4)$ & $\mathfrak{so}_{2n}$ &  $n(2n-1)$ & $n$ &     $2 n(n-1)$         & $\dynkin{D}{}$ \\
    % G_2 & \mathfrak{g}_2 &  14  & 2 &  12                              & \dynkin{G}{2}\\
    % F_4 & \mathfrak{f}_4 &  52  & 4 &  48                    & \dynkin{F}{4}\\
    $E_6$ & $\mathfrak{e}_6$ &  $78$ & $6$ &  $72$              & $\dynkin{E}{6}$\\
    $E_7$ & $\mathfrak{e}_7$ &  $133$  & $7$ &  $126$      & $\dynkin{E}{7}$\\
    $E_8$ & $\mathfrak{e}_8$ &  $248$  & $8$ &  $240$  & $\dynkin{E}{8}$\\
    \hline\hline
  \end{tabular}
  \caption{The Lie algebras of simply-laced root systems.}
  \label{simpleLieAlgebras}
\end{figure}

Suppose that indecomposable star-closed line system $\overline{\Phi}$ admits a binary grading, $\overline{\Phi} = \overline{\Phi}_{{1}} ~\dot{\cup}~\overline{\Phi}_{{0}}$. 
The lines of the star-free component $\overline{\Phi}_1$ can be spanned by roots with non-negative inner products. 
We denote these spanning roots by $\Phi_1$ and define $\Phi_{-1} = - \Phi_1$ as the set of opposite roots, which also has all non-negative inner products. 
Then we have the following $3$-grading on $\mathfrak{g}$ as a coarsening of the Cartan grading:
\begin{align*}
    \mathfrak{g} = \left( \bigoplus_{r \in \Phi_{-1}} \mathfrak{g}_r\right) \oplus \left(\mathfrak{h} \oplus \bigoplus_{r \in \Phi_0} \mathfrak{g}_r\right)
    \oplus 
    \left( \bigoplus_{r \in \Phi_{1}} \mathfrak{g}_r\right) = \mathfrak{g}(-1) \oplus \mathfrak{g}(0) \oplus \mathfrak{g}(1).
\end{align*}
That is,
\begin{align*}
    [\mathfrak{g}(i),\mathfrak{g}(j)] \subseteq \mathfrak{g}(i+j). 
\end{align*}
We see, then, that a $3$-grading on a root system $\Phi \xrightarrow{n} \Phi_0$ defines abelian Lie subalgebras $\mathfrak{g}(-1)$ and $\mathfrak{g}(1)$ of dimension $n = |\Phi_{-1}| = |\Phi_{1}| = |\overline{\Phi}_1|$. 
The $\mathfrak{g}(0)$ Lie subalgebra acts on each of these abelian Lie subalgebras via $[\mathfrak{g}(0), \mathfrak{g}(\pm 1)] \subseteq \mathfrak{g}(\pm 1)$.
Also, since the entire Cartan subalgebra $\mathfrak{h}$ is contained in $\mathfrak{g}(0)$, we see that $\mathfrak{g}(0)$ is not isomorphic to the Lie algebra constructed from root system $\Phi_0$, but rather is the direct product of this algebra and the one-dimensional abelian Lie algebra:
\begin{align*}
    \mathfrak{g}(0) = \mathbb{C} \oplus [\mathfrak{g}(0), \mathfrak{g}(0)].
\end{align*}
That is, $\mathfrak{g}(0)$ contains $\mathfrak{h}$, the Cartan subalgebra of $\mathfrak{g}$. But $[\mathfrak{g}(0), \mathfrak{g}(0)]$ does not contain $\mathfrak{h}$. The Cartan subalgebra of $[\mathfrak{g}(0), \mathfrak{g}(0)]$ is a subalgebra of $\mathfrak{h}$ with one dimension less than $\mathfrak{h}$.
% the Cartan subalgebra of $\mathfrak{g}$. 

In particular, the binary decomposition $A_4 \rightarrow A_1 \times A_2$ signifies the following Lie algebra $3$-grading:
\begin{align*}
    \mathfrak{sl}_5 = \mathfrak{sl}_5(-1) \oplus \mathfrak{sl}_5(0) \oplus \mathfrak{sl}_5(1),
\end{align*}
where $\mathfrak{sl}_5(1)$ is six-dimensional and $[\mathfrak{sl}_5(0),\mathfrak{sl}_5(0)] = \mathfrak{sl}_2 \oplus \mathfrak{sl}_3$.
This means that the $0$-piece of this $3$-grading is,
\begin{align*}
    \mathfrak{sl}_5(0) = \mathbb{C} \oplus \mathfrak{sl}_2 \oplus \mathfrak{sl}_3.
\end{align*}
This Lie algebra---the $0$-piece of the 3-grading due to the $A_4 \rightarrow A_1\times A_2$ binary decomposition---is in fact the Lie algebra of the standard model of particle physics.

\section{Connection to the Standard Model}

The exceptional sequence \eqref{mainsequence} corresponds to the following sequence of nested Lie algebra $3$-gradings:
\begin{align*}
    \mathfrak{e}_7 \xrightarrow{27} 
    \mathfrak{e}_6 \xrightarrow{16}
    \mathfrak{so}_{10} \xrightarrow{10} 
    \mathfrak{sl}_{5} \xrightarrow{6} 
    \mathfrak{sl}_2 \oplus \mathfrak{sl}_3.
\end{align*}
The final arrow, $A_4 \rightarrow A_1 \times A_2$, corresponding to the diagram $\dynkin{A}{**o*}$, yields the Lie algebra of the standard model of particle physics as the $0$-piece of the 3-grading.
\begin{align*}
    \mathfrak{g}_{SM} = \mathbb{C} \oplus \mathfrak{sl}_2 \oplus \mathfrak{sl}_3.
\end{align*}
Our next step is to determine the action of $\mathfrak{g}_{SM}$ on the rest of $\mathfrak{e}_7$, so that we can identify certain root spaces with familiar standard model particles.

Each root in $E_7$ indexes a one-dimensional root space $\mathfrak{g}_r$, spanned by vector $e_r$, in the Lie algebra $\mathfrak{e}_7$ described above.
By construction, each $e_r$ is an eigenvector of each $h$ in the Cartan subalgebra $\mathfrak{h}$, since we have $[h_s, e_r] = 2 \langle s, r\rangle /\langle s, s\rangle e_r = \langle s, r \rangle e_r$.
Recall that (in the Chevalley basis) $\mathfrak{g}_{SM}$ contains the Cartan subalgebra of $\mathfrak{sl}_5$, which has dimension $4$ and is itself a subalgebra of $\mathfrak{h}$, the Cartan subalgebra of $E_7$. 
In order to find the correspondence between root spaces $\mathfrak{g}_r$ and particles, we need to find a well-chosen basis of $\mathfrak{h} \cap \mathfrak{g}_{SM}$ (the Cartan subalgebra of $\mathfrak{sl}_5$).
The four simultaneous eigenvalues with respect to this basis give us the familiar \textit{hypercharge}, \textit{isospin}, and \textit{colour} of each particle (where colour signifies a pair of eigenvalues). 
Since $\mathfrak{h}$ is seven dimensional, there are three possible remaining simultaneous eigenvalues. 
We can use two of these to assign a \textit{generation} to each root space $\mathfrak{g}_r$ and the remaining eigenvalue to distinguish particles of the standard model from additional particles.

For specificity, we will denote the exceptional sequence \eqref{mainsequence} in terms of Coxeter-Dynkin diagrams as follows:
\begin{align*}
    \dynkin{E}{*******} \xrightarrow{} 
    \dynkin{E}{******o} \xrightarrow{} 
    \dynkin{E}{*****oo} \xrightarrow{} 
    \dynkin{E}{****ooo} \xrightarrow{} 
    \dynkin{E}{***oooo} 
\end{align*}
We may write vectors in the Cartan subalgebra $\mathfrak{h} \subset \mathfrak{e}_7$ using Dynkin diagrams, e.g.:
\begin{align*}
    \dynkin[labels = {a_1, a_2, a_3, a_4, a_5, a_6, a_7}]{E}{7} = \sum_{i = 1}^7 a_i h_{s_i} \in \mathfrak{h},
\end{align*}
where $s_i$ are a set of simple roots of $E_7$. 
This means that we compute the eigenvalues of the action of a vector in $\mathfrak{h}$ on a root space as follows:
\begin{align*}
    \left[ \dynkin[labels = {a_1, a_2, a_3, a_4, a_5, a_6, a_7}]{E}{7} ,\mathfrak{g}_r\right] = \left(\sum_{i = 1}^7 a_i (s_i \cdot r) \right) \mathfrak{g}_r.
\end{align*}

We define the \textbf{isospin} of each root space $\mathfrak{g}_r$ as its eigenvalue for multiplication by the following vector in $\mathfrak{h}\cap \mathfrak{g}_
{SM}$:
\begin{align*}
    W_0 &= \dynkin[labels = {0,\frac{1}{2},0,0,0, 0,0}]{E}{7}.
\end{align*}
The vector $2 W_0$ is a coweight of $E_7$ (also a coroot) and defines an isospin $5$-grading on $E_7$. Root spaces with isospin $0$ correspond to \textbf{right-handed} particles (left-handed anti-particles). Root spaces with isospin $\pm \frac{1}{2}$ correspond to \textbf{left-handed} particles (right-handed anti-particles). The unique root spaces with isospins $\pm 1$ correspond to the $W^{\pm}$ bosons. Specifically, $W_0$ and $W^\pm$ span the $\mathfrak{sl}_2$ (i.e., $A_1$), component of  the standard model Lie algebra $\mathfrak{g}_{SM}$.

We define the \textbf{colour} of each root space as the pair of eigenvalues of the following vectors in $\mathfrak{h}\cap \mathfrak{g}_{SM}$: 
\begin{align*}
    \lambda_3 = \dynkin[labels = {1, 0, 0, 0, 0, 0, 0}]{E}{7},
    &&
    \sqrt{3}\lambda_8 = \dynkin[labels = {1, 0, 2, 0, 0, 0, 0}]{E}{7}.
\end{align*}
The vectors $\lambda_3$, $\sqrt{3}\lambda_8$ in $\mathfrak{h}$ as well as the six unique root spaces $\mathfrak{g}_r$ with eigenvalues $\pm (2,0)$, $\pm (-1,3)$, $\pm (-1,-3)$ form the $\mathfrak{sl}_3$ (i.e., $A_2$) component of the standard model Lie algebra $\mathfrak{g}_{SM}$. These eight-dimensions of $\mathfrak{e}_7$ represent the eight \textbf{gluons}, the bosons responsible for the strong force. 
The corresponding $A_2$ root system defines a star-decomposition of $E_7$ that allows us to assign particle colour.
Specifically, we will call \textbf{blue} the fifteen root spaces with $\lambda_3, \sqrt{3}\lambda_8$ eigenvalues $(0,2)$, their opposite root spaces are called \textbf{anti-blue}.
Likewise, eigenvalues $(-1,-1)$ signify \textbf{red} and eigenvalues $(1,-1)$ signifies \textbf{green}. The opposite eigenvalues signify \textbf{anti-red} and \textbf{anti-green}.
Finally, the $30$ root spaces with eigenvalues $(0,0)$ are called \textbf{colourless}. 
Root spaces outside of $\mathfrak{g}_{SM}$ that are red, green, or blue correspond to \textbf{quarks} whereas those that are colourless correspond to \textbf{leptons}. 

We define the \textbf{hypercharge} of each root space as the eigenvalue of the following operator:
\begin{align*}
    B &= \dynkin[labels = {\frac{2}{3}, 1, \frac{4}{3}, 2, 0, 0, 0}]{E}{7}.
\end{align*}
The 3-grading defined by $\dynkin{E}{****ooo} \xrightarrow{} \dynkin{E}{***oooo}$ also defines a unique line perpendicular to the $\mathbb{R}^3$ of $A_1\times A_2$ (spanned by $W_0$, $\lambda_3$, $\lambda_8$) but also within the $\mathbb{R}^4$ spanned by the coroots of $A_4$. 
This unique line is spanned by the hypercharge operator $B$.
Although $B$ is not a coweight, there is a $13$-grading defined by coweight $3B$. 
This means that the eigenvalues of $B$ are in the set $\left\{0, \pm \frac{1}{3}, \pm \frac{2}{3}, \pm 1, \pm \frac{4}{3}, \pm \frac{5}{3}, \pm 2\right\}$.
All of these values correspond to known physical particles except for $\pm \frac{5}{3}$, which are the eigenvalues of the root spaces of the roots in $A_4 \setminus (A_1\times A_2)$.

\begin{figure}
    \centering
    \begin{align*}
\def\arraystretch{1.2}
\begin{array}{lcrrrr}%
\hline \hline
\text{Name} & \text{Symbol} & B & W_0 & \lambda_3 & \sqrt{3}\lambda_8 \\
\hline 
\text{Right-handed neutrino} & \nu_R & 0 & 0 & 0 & 0 \\
\hline
\text{Right-handed electron} & e^{-}_R & -2 & 0 & 0 & 0 \\
\hline
\text{Right-handed red up quark} & u_R^{r} & \frac{4}{3} & 0 & -1 & -1 \\
\text{Right-handed green up quark} & u_R^{g} & \frac{4}{3} & 0 & 1 & -1 \\
\text{Right-handed blue up quark} & u_R^{b} & \frac{4}{3} & 0 & 0 & 2 \\
\hline
\text{Right-handed red down quark} & d_R^{r} & -\frac{2}{3} & 0 & -1 & -1 \\
\text{Right-handed green down quark} & d_R^{g} & -\frac{2}{3} & 0 & 1 & -1 \\
\text{Right-handed blue down quark} & d_R^{b} & -\frac{2}{3} & 0 & 0 & 2 \\
\hline
\text{Left-handed neutrino} & \nu_L & -1 & \frac{1}{2} & 0 & 0 \\
\text{Left-handed electron} & e^{-}_L & -1 & -\frac{1}{2} & 0 & 0 \\
\hline
\text{Left-handed red up quark} & u_L^{r} & \frac{1}{3} & \frac{1}{2} & -1 & -1 \\
\text{Left-handed green up quark} & u_L^{g} & \frac{1}{3} & \frac{1}{2} & 1 & -1 \\
\text{Left-handed blue up quark} & u_L^{b} & \frac{1}{3} & \frac{1}{2} & 0 & 2 \\
\text{Left-handed red down quark} & d_L^{r} & \frac{1}{3} & -\frac{1}{2} & -1 & -1 \\
\text{Left-handed green down quark} & d_L^{g} & \frac{1}{3} & -\frac{1}{2} & 1 & -1 \\
\text{Left-handed blue down quark} & d_L^{b} & \frac{1}{3} & -\frac{1}{2} & 0 & 2 \\
\hline
\hline
\end{array}
\end{align*}
\caption{Fermion particle nomenclature.}
    \label{particlenomenclature}
\end{figure}

Using the four simultaneous eigenvalues of $B, W_0, \lambda_3, \lambda_8$, we can assign a standard particle name to each of the root spaces $\mathfrak{g}_r$ with roots in $E_7 \setminus A_4$, as shown in Fig. \ref{particlenomenclature}. Here we label particles according to the eigenvalues for hypercharge and isospin given in \cite{baez_algebra_2010}, while the three colour labels (red, green, blue) are treated as conventional. 

\begin{remark}
    \textbf{Anti-particles} correspond to roots with opposite eigenvalues of the partner particle. 
    Just as each root describes a particle or anti-particle, each line in the corresponding line system describes a particle/anti-particle pair. 
    Whether we choose to work with Lie structures (roots) or Jordan structures (lines) largely corresponds to whether we choose to work with particles or with particle/anti-particle pairs.
\end{remark}

The next task is to sort the particles into generations, and to identify any additional particles beyond those given in the standard model. To do so, we note that the Lie centralizer of the standard model Lie algebra in $\mathfrak{e}_7$ has the form,
\begin{align*}
    C_{\mathfrak{e}_7}(\mathfrak{g}_{SM}) = \mathbb{C}^2 \oplus \mathfrak{sl}_3.
\end{align*}
The $\mathfrak{sl}_3$ component is generated by the root spaces $\mathfrak{g}_r$ corresponding to the unique six roots in $E_7$ perpendicular to each root in $A_4$. 
These six root are unique in $E_7$ in that their root spaces have null hypercharge, isospin, and are colourless. 
For this reason, we call them \textbf{right-handed neutrinos} (and left-handed anti-neutrinos)---the undetectable partners to left-handed neutrinos (and right-handed anti-neutrinos). 
These six $\mathfrak{g}_r$ root spaces in the centralizer of $\mathfrak{g}_{SM}$ serve the same role as the six coloured gluons in $\mathfrak{g}_{SM}$. 
Just as the six coloured gluons define the star-decomposition of $E_7$ that gives us particle colour, the three right-handed neutrinos and their anti-particles can be used to define a second star-decomposition of $E_7$ that gives us particle \textbf{generation}. 
We assign particle generation to each root space $\mathfrak{g}_r$ using the eigenvalue pair of the following two operators:
\begin{align*}
    \rho_3 = \dynkin[labels = {0, 0, 0, 0, 0, 1, 0}]{E}{7},
    &&
    \sqrt{3}\rho_8 = \dynkin[labels = {0, 0, 0, 0, 0, 1, 2}]{E}{7}.
\end{align*}
Specifically, we may call the thirty root spaces with $\rho_3, \sqrt{3}\rho_8$ eigenvalues $\pm (0,2)$  the \textbf{first generation}, the thirty with eigenvalues $\pm (1,1)$ the \textbf{second generation}, and the thirty with eigenvalues $\pm (1,-1)$ the \textbf{third generation}.
Each generation consists of fifteen particles with the eigenvalues given in Fig. \ref{particlenomenclature} and the corresponding fifteen anti-particles.
Any root spaces with eigenvalues $(0,0)$ do not belong to any generation. 
These include the boson root spaces of $\mathfrak{g}_{SM}$ and $22$ additional root spaces. 

So far we have defined an orthogonal basis $\{\rho_3, \rho_8, B, W_0, \lambda_3, \lambda_8\}$ for a $\mathbb{C}^6$ subspace of $\mathfrak{h}$, and can use the simultaneous eigenvalues of this basis to partition $\mathfrak{e}_7$ into the familiar standard model bosons $\mathfrak{g}_{SM}$, a right-handed neutrino $\mathfrak{sl}_3$, three generations of fifteen particles and their anti-particles, plus $22$ additional root spaces and one remaining dimension of $\mathfrak{h}$ perpendicular to this $\mathbb{C}^6$.
We can use this remaining dimension to distinguish familiar particles from potentially new and unobserved ones. 
That is, we define a seventh vector in $\mathfrak{h}$ perpendicular to $\mathbb{C}^6$:
\begin{align*}
    H &= \dynkin[labels = {1, \frac{3}{2}, 2, 3, \frac{5}{2}, \frac{5}{3}, \frac{5}{6}}]{E}{7}
\end{align*}
The vector $3H$ is a coweight of $E_7$ and defines a $7$-grading, so the eigenvalues of $H$ are in the set $\left\{0, \pm \frac{1}{3}, \pm \frac{2}{3}, \pm 1\right\}$.

It turns out that the three generations of particles are precisely the root spaces $\mathfrak{g}_r$ with $H$ eigenvalues $\pm \frac{1}{3}$ and $\pm \frac{2}{3}$. 
Furthermore the particles with $H$ eigenvalue $0$ consist of the bosons of $\mathfrak{g}_{SM}$, the right-handed neutrino $\mathfrak{sl}_3$, and the particles with hypercharge $\pm \frac{5}{3}$ (corresponding to root spaces $\mathfrak{g}_r$ with $r$ in $A_4 \setminus A_1\times A_2$). 

To summarize, we can trim out the unobserved particles of $\mathfrak{e}_7$ by making $\pm 1$ a forbidden eigenvalue of $H$ and $\pm \frac{5}{3}$ a forbidden eigenvalue of $B$. All other root spaces correspond to familiar bosons and the three generations of fermions. 

\begin{remark}
    The fifteen particle/anti-particle pairs of a single generation correspond to a generalized quadrangle structure in the following way. 
    If we take the corresponding $15$ roots in $E_7$, then these span a star-free line system representing the unique graph $\mathrm{srg}(15,8,4,4)$. 
    This graph has precisely $15$ maximal independent sets, all of size $3$, representing triads of orthogonal lines. These triads serve as the ``lines'' of a generalized quadrangle $GQ(2,2)$. 
    In terms of the nomenclature given in Fig. \ref{particlenomenclature}, the particle triads are:
    \begin{align*}
        \begin{matrix}
            \{ \nu_L, u_L^r, u_R^r  \}, &
            \{ \nu_L, u_L^g, u_R^g \}, &
            \{ \nu_L, u_L^b, u_R^b \}, \\
            \{ e_L^{-}, d_L^r, u_R^r \}, &
            \{ e_L^{-}, d_L^g, u_R^g \}, &
            \{ e_L^{-}, d_L^b, u_R^b \}, \\
            \{ e_R^{-}, u_R^r, d_R^r \}, &
            \{ e_R^{-}, u_R^g, d_R^g \}, &
            \{ e_R^{-}, u_R^b, d_R^b \}, \\
            \{ u_L^r, d_L^g, d_R^b  \}, &
            \{ u_L^r, d_L^b, d_R^g  \}, &
            \{ u_L^g, d_L^r, d_R^b  \}, \\
            \{ u_L^b, d_L^r, d_R^g \}, &
            \{ u_L^g, d_L^b, d_R^r \}, &
            \{ u_L^b, d_L^g, d_R^r  \}. \\
        \end{matrix}
    \end{align*}  
  Of these fifteen triads, six have the property that they do not contain a lepton. These six are also the only six where the eigenvalues of $B, W_0, \lambda_3, \lambda_8$ each add to zero over the triad:
  \begin{align*}
      \begin{matrix}
        \{ u_L^r, d_L^g, d_R^b  \}, & 
        \{ u_L^r, d_L^b, d_R^g  \}, &
        \{ u_L^g, d_L^r, d_R^b  \}, \\ 
        \{ u_L^b, d_L^r, d_R^g \}, &
        \{ u_L^g, d_L^b, d_R^r \}, &
        \{ u_L^b, d_L^g, d_R^r  \}.
      \end{matrix}
  \end{align*}
  In fact, this subset of six triads forms a smaller generalized quadrangle $GQ(2,1)$ on nine points. 
  The roots corresponding to these particle root spaces have the following interesting property. The roots of the $GQ(2,1)$ particles all have non-negative inner product, as do the roots of the $GQ(2,2)\setminus GQ(2,1)$ particles. However, the inner products between a root from each of the two sets is always non-positive. 
  The fact that a generation of $15$ particles does not correspond to a set of roots with all non-negative inner products, but rather describes an embedding of $G(2,1)$ within $G(2,2)$, leaves a tempting combinatorial clue regarding abundance of matter and the dearth of antimatter in the physical universe. 
\end{remark}

\section{Discussion}

This note does not attempt to account for the Higgs mechanism, the embedding of electromagnetism within the electroweak force, or particle spin. Neither does it speculate on a role for the $22$ additional root-spaces within $\mathfrak{e}_7$ that do not correspond to familiar particles of the standard model. 
Rather, this note attempts to convert certain questions about the accidental properties of particle physics into corresponding questions about exceptional mathematical objects. 
To the question of why we have this particular standard model Lie algebra $\mathfrak{g}_{SM}$ and not another, perhaps we could answer that this is the Lie algebra in which the exceptional sequence terminates. To the question of why there are three generations of fifteen (or sixteen) particles that represent this Lie algebra, perhaps we could answer that the exceptional sequence defines an action of $\mathfrak{g}_{SM}$ on $\mathfrak{e}_7$ and that star-decomposition explains the existence of three generations. 
Most remarkably, questions about physical symmetries and structures can perhaps be answered in terms of systems of equiangular lines at the absolute bound, beginning with 3 line stars and the 28 lines spanned by the minuscule coweights of $E_7$.

\bibliographystyle{amsalpha}
\bibliography{references}

\appendix

\end{document}